\documentclass[11pt,a4paper,oneside]{amsart}  

\usepackage{amsmath}
\usepackage{amsthm}
\usepackage{amsfonts}
\usepackage{dsfont}
\usepackage{enumitem}
\usepackage{color}
\usepackage[margin=36mm]{geometry}
\usepackage{graphicx}

\newcommand{\reals}{\mathbb{R}}

\newcommand{\complex}{\mathbb{C}}

\newcommand{\bracketa}[1]{\big[#1\big]}

\newcommand{\angles}[1]{\left\langle #1 \right\rangle}

\newcommand{\paraa}[1]{\big(#1\big)}
\newcommand{\parab}[1]{\Big(#1\Big)}
\newcommand{\parac}[1]{\bigg(#1\bigg)}

\newcommand{\diag}{\operatorname{diag}}

\newcommand{\spacearound}[1]{\quad#1\quad}
\newcommand{\equivalent}{\spacearound{\Leftrightarrow}}

\newcommand{\qtext}[1]{\quad\text{#1}\quad}

\newcommand{\qand}{\qtext{and}}

\newcommand{\mathand}{\qand}

\newtheorem{theorem}{Theorem}[section]
\newtheorem{corollary}[theorem]{Corollary}
\newtheorem{lemma}[theorem]{Lemma}
\newtheorem{proposition}[theorem]{Proposition}

\theoremstyle{definition}
\newtheorem{definition}[theorem]{Definition}
\theoremstyle{remark}
\newtheorem{remark}[theorem]{Remark}
\numberwithin{equation}{section}

\newcommand{\half}{\frac{1}{2}}

\renewcommand{\d}{\partial}
\newcommand{\eps}{\varepsilon}

\newcommand{\T}{\mathcal{T}}

\newcommand{\rh}{r_{\hbar}}
\newcommand{\qh}{q_{\hbar}}
\newcommand{\ph}{\hat{p}}
\newcommand{\Aeq}[1]{\textsf{A}$_{#1}$}
\newcommand{\Beq}[1]{\textsf{B}$_{#1}$}
\newcommand{\Cgh}{\mathcal{C}_{\hbar}^g}
\newcommand{\xh}{\hat{x}}

\newcommand{\fh}{f_{\hbar}}

\newcommand{\Mat}{\operatorname{Mat}}
\newcommand{\C}{\mathcal{C}}

\title[]{Low dimensional matrix representations for noncommutative surfaces of arbitrary genus}
\author{Joakim Arnlind}

\address[Joakim Arnlind]{Dept. of Math.\\
Link\"oping University\\
581 83 Link\"oping\\
Sweden}
\email{joakim.arnlind@liu.se}

\begin{document}

\begin{abstract}
  In this note, we initiate a study of the finite-dimensional
  representation theory of a class of algebras that correspond to
  noncommutative deformations of compact surfaces of arbitrary
  genus. Low dimensional representations are investigated in detail
  and graph representations are used in order to understand the
  structure of non-zero matrix elements. In particular, for arbitrary
  genus greater than one, we explicitly construct classes of
  irreducible two and three dimensional representations. The existence
  of representations crucially depends on the analytic structure of the
  polynomial defining the surface as a level set in $\reals^3$.
\end{abstract}

\maketitle

\section{Introduction}

\noindent
Understanding the geometry of noncommutative space is believed to be
crucial in order to approach a quantum theory of gravity. Both String
theory (via the IKKT model \cite{ikkt:superstring}) and the matrix
regularization of Membrane theory
\cite{h:phdthesis,dwhn:supermembranes}, being candidates for
describing quantum effects in gravity, contain noncommutative (matrix)
analogues of 2-dimensional manifolds. For compact surfaces, one
considers sequences of matrix algebras of increasing dimension,
converging (in a certain sense
\cite{h:phdthesis,bhss:glinfinity,ahh:multilinear}) to the Poisson
algebra of functions on the surface.  By now, surfaces of genus zero
and one are quite well understood (see
e.g. \cite{h:phdthesis,ffz:trigonometric,h:diffT,h:diffeomorphism,m:fuzzy.sphere}),
but understanding the case of higher genus has turned out to be more
difficult. Although there are several results that treat the case of
higher genus and prove the existence of matrix algebras converging to
the algebra of smooth functions (see
e.g. \cite{kl:quantumriemannI,kl:quantumriemannII,kl:quantumriemannIII,bms:toeplitz,nn:quantum.surfaces}),
explicit representations, as well as an algebraic understanding of
these objects, are still lacking to a large extent. Other interesting
approaches to matrix regularizations have also appeared which focus
slightly more on approximation properties, as well as connections to
physics, and how geometry emerges from limits of matrix algebras (see
e.g. \cite{h:membrane.topology,s:emergent.geometry,s:construction.kit}).

In \cite{abhhs:fuzzy,abhhs:noncommutative} a class of algebras,
defined by generators and relations, was given as a candidate for
noncommutative analogues of compact surfaces of arbitrary genus. A
one-parameter family of surfaces interpolating between spheres and
tori was thoroughly investigated and all finite-dimensional
representations were classified. However, only marginal progress was
made in understanding if the proposed relations are consistent and
tractable for the higher genus case, and no concrete representations
were constructed. In this note, we study low dimensional
representations of these algebras for arbitrary genus greater than
one.  One should expect, due to the high polynomial order of the
defining relations, that the representation theory is quite
complicated and to better understand its structure, we shall make use
of graph methods to describe non-zero matrix elements; a method which
has previously turned out to be most helpful in understanding
finite-dimensional representations
\cite{a:phdthesis,a:repcalg,abhhs:fuzzy,abhhs:noncommutative,ah:dmsa}. Via
these graphs, one can easily derive conditions which may be used to
exclude certain matrices from being representations.

Two dimensional representations are studied in detail, and even this
simple case turns out to be rather complicated, indicating what
to expect in the higher dimensional case. In particular, we show that
one can construct 2-dimensional representations for any genus
$g\geq 2$ and any value of the deformation parameter $\hbar>0$. The
existence of these representations depends crucially on the analytic
structure of a polynomial defining the surface as a level set in
$\reals^3$.

\section{Compact genus $g$ surfaces in $\reals^3$ as level sets}

\noindent
Let us recall how a class of compact surfaces of arbitrary genus may
be constructed as level sets in $\reals^3$
\cite{hofer:phdthesis,abhhs:fuzzy,abhhs:noncommutative}. Let $g\geq 1$
be an integer and set
\begin{align*}
  G(t) = \prod_{k=1}^g(t-k^2)\qquad
  M=\!\!\!\!\max_{0\leq t\leq g^2+1}\!\!\!G(t).
\end{align*}
For arbitrary $c>0$ and $\alpha\in(0,2\sqrt{c}/M)$ set
\begin{align*}
  p(x) = \alpha G(x^2)-\sqrt{c}
  =\alpha\prod_{k=1}^g(x^2-k^2)-\sqrt{c}
\end{align*}
and
\begin{align*}
  C(x,y,z) = \half\paraa{p(x)+y^2}^2+\half z^2-\half c.
\end{align*}
Using Morse theory is is straightforward to show that the level set
$\Sigma_g=C^{-1}(0)$ is a compact surface of genus $g$. An example of
a surface of genus 3 is given in Figure~\ref{fig:genus.3}.  In the
next section, we shall start by investigating some of the properties
of the polynomial $p(x)$, which will later become important when
proving existence of representations.

\begin{figure}[h]
  \centering
  \includegraphics[height=4cm]{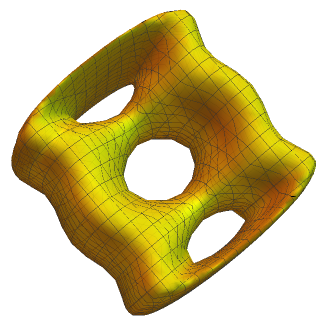}
  \caption{An example of $\Sigma_3$ as a level set in $\reals^3$.}
  \label{fig:genus.3}
\end{figure}

\subsection{Properties of $p(x)$}

First of all, one notes that
\begin{align*}
  p(k) = -\sqrt{c}<0
\end{align*}
for $k=1,2,\ldots,g$ and
\begin{align*}
  &p'(x) = 2\alpha x\sum_{k=1}^g\prod_{l=1,l\neq k}^g(x^2-l^2)
\end{align*}
giving
\begin{align}\label{eq:pp.integer.value}
  p'(k) = 2\alpha k\prod_{l=1,l\neq k}^g(k^2-l^2)
  =(-1)^{g-k}\,\frac{\alpha(g+k)!(g-k)!}{k}
\end{align}
for $k=1,2,\ldots,g$. Furthermore, $p(-x)=x$ and $p'(-x)=-p'(x)$. As
an illustration of the typical behavior of the polynomial, a plot of
$p(x)$ can be found in Figure~\ref{fig:px.g4}.

\begin{figure}[h]
  \centering
  \includegraphics[width=7cm]{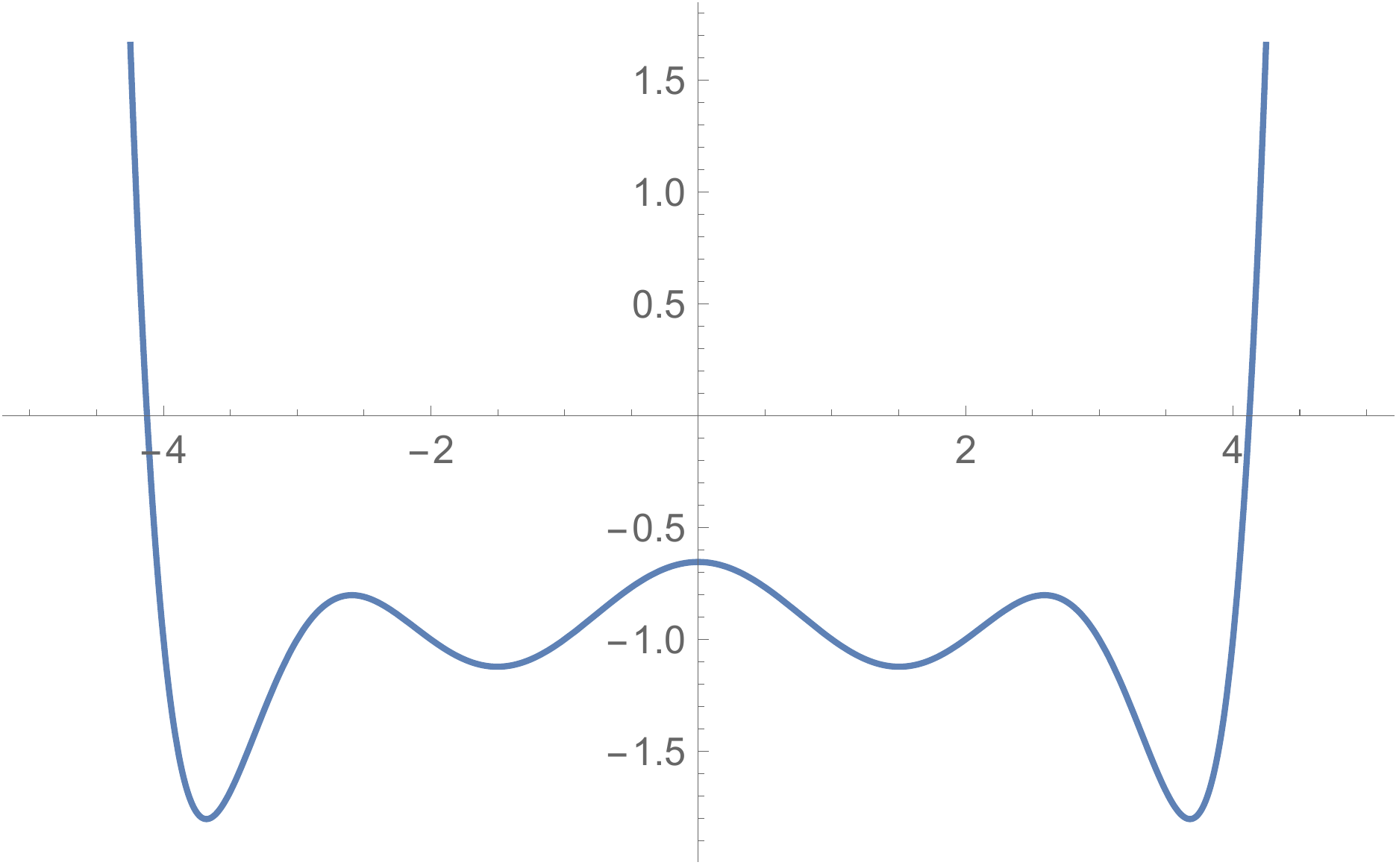}
  \caption{Plot of the polynomial $p(x)$ for $g=4$ ($c=1$ and
    $\alpha=1/1664$).}
  \label{fig:px.g4}
\end{figure}

\noindent
In the next result we give two simple, but useful, lower bounds of the
maximum of $G(t)$ in the interval $0\leq t\leq g^2+1$.

\begin{lemma}\label{lemma:M.lower.bound}
  If
  \begin{align*}
    G(t) = \prod_{k=1}^g(t-k^2)\mathand
    M=\!\!\!\!\max_{0\leq t\leq g^2+1}\!\!\!G(t)
  \end{align*}
  then $M\geq(2g-1)!/g$ and $M\geq (g!)^2$.
\end{lemma}

\begin{proof}
  First we note that
  \begin{align*}
    M&\geq G(g^2+1) = \prod_{k=1}^g(g^2+1-k^2)
    \geq\prod_{k=1}^{g-1}(g^2-k^2)
    =\prod_{k=1}^{g-1}(g+k)(g-k)\\
    &=\frac{(2g-1)!}{g},
  \end{align*}
  proving the first inequality. Next, for $g\leq 4$ one verifies that
  \begin{align*}
    &g=1:\qquad M\geq G(g^2+1) = 1 = (1!)^2\\
    &g=2:\qquad M\geq G(g^2+1) = 4 = (2!)^2\\
    &g=3:\qquad M\geq G(g^2+1) = 54>36=(3!)^2\\
    &g=4:\qquad M\geq G(g^2+1) = 1664>576 = (4!)^2.
  \end{align*}
  Now, assume $g\geq 5$ and write
  \begin{align*}
    M-(g!)^2&\geq \frac{(2g-1)!}{g}-(g!)^2\\
            &=(g-1)!\parab{(g+1)(g+2)\cdots(2g-1)-g\cdot g(g-1)\cdots 2\cdot 1}\\
            &=(g-1)!\parab{(g+1)(g+2)\cdots(2g-1)-g\cdot g(g-1)\cdots 5\cdot 4\cdot 6}.
  \end{align*}
  As written, each product inside the parenthesis has $g-1$ factors,
  and every factor in the positive term is $\geq g+1$ and every factor
  in the negative term is $\leq g+1$ (since $g\geq 5$). Thus, we
  conclude that $M-(g!)^2\geq 0$.
\end{proof}

\noindent
With the help of the above lemma, one can prove the following result.

\begin{lemma}\label{lemma:p.zero.c}
  For $\alpha\in(0,2\sqrt{c}/M)$ it holds that
  \begin{align*}
    &p(0)+3\sqrt{c}>0\mathand
    p(0)-\sqrt{c}<0.
  \end{align*}
\end{lemma}

\begin{proof}
  The two inequalities can be written as
  \begin{align*}
    &(-1)^g\alpha(g!)^2+2\sqrt{c} > 0\\
    &(-1)^g\alpha(g!)^2-2\sqrt{c} < 0,
  \end{align*}
  and they are trivially satisfied if $g$ is even and odd,
  respectively. In the opposite case, these inequalities are both
  equivalent to
  \begin{align*}
    \alpha(g!)^2<2\sqrt{c}\equivalent
    \frac{\alpha}{\sqrt{c}}<\frac{2}{(g!)^2}.
  \end{align*}
  Recalling that $\alpha<2\sqrt{c}/M$, Lemma~\ref{lemma:M.lower.bound}
  implies that
  \begin{align*}
    \frac{\alpha}{\sqrt{c}}<\frac{2}{M}\leq \frac{2}{(g!)^2}
  \end{align*}
  yielding the desired result.
\end{proof}

\section{Noncommutative surfaces of arbitrary genus}

\noindent
In this section we will recall a class of noncommutative algebras
corresponding to deformations of algebras of smooth functions on
the level sets $\Sigma_g=C^{-1}(0)$. For arbitrary
$C\in C^\infty(\reals^3)$ the relation
\begin{align*}
  \{x^i,x^j\} = \eps^{ijk}(\d_kC)
\end{align*}
(with $x^1=x$, $x^2=y$, $x^3=z$) defines a Poisson structure on
$C^{\infty}(\reals^3)$ that restricts to the level set $C(x,y,z)=0$. For the
particular case when
\begin{align*}
  C(x,y,z) = \frac{1}{2}\paraa{p(x)+y^2}^2+\frac{1}{2}z^2-\frac{1}{2}c
\end{align*}
one obtains
\begin{align*}
  &\{x,y\} = z\\
  &\{y,z\} = p(x)p'(x)+y^2p'(x)\\
  &\{z,x\} = 2y^3+2yp(x). 
\end{align*}
In order to find noncommutative deformations of the above Poisson
algebra, one replaces $\{\cdot,\cdot\}$ by $[\cdot,\cdot]/(i\hbar)$
and considers the relations (cf. \cite{abhhs:noncommutative})
\begin{align}
  &[X,Y] = i\hbar Z\label{eq:XYcom}\\
  &[Y,Z] = i\hbar\paraa{p(X)p'(X)+\mathcal{T}(X,Y^2)}\label{eq:YZcom}\\
  &[Z,X] = i\hbar\paraa{2Y^3+Yp(X)+p(X)Y}\label{eq:ZXcom}
\end{align}
where $\T(X,Y^2)$ denotes a noncommutative polynomial such that its
commutative image $\T(x,y^2)$ equals $p'(x)y^2$; for instance
\begin{align*}
 \T\paraa{X,Y^2}=\tfrac{1}{2}Y^2p'(X)+\tfrac{1}{2}p'(X)Y^2 .
\end{align*}
In other words, $\T$ represents a choice of noncommutative ordering of
the product of $y^2$ and $p'(x)$. Let us keep this choice arbitrary
for the moment and define the noncommutative algebra that will
be of interest for us.

\begin{definition}\label{def:Cgh.def}
  For $g\geq 1$, $c>0$, $\alpha\in(0,2\sqrt{c}/M)$ and $\hbar>0$, let
  $I_\hbar^g(\alpha,c)$ denote the two-sided ideal, in the (unital) free
  $\ast$-algebra $\complex\angles{X,Y,Z}$ (with $X,Y,Z$
  hermitian), generated by \eqref{eq:XYcom}--\eqref{eq:ZXcom}. We set
  \begin{align*}
    \Cgh(\alpha,c)=\complex\angles{X,Y,Z}\slash I_\hbar^g(\alpha,c).
  \end{align*}
\end{definition}

\begin{remark}
  Note that we shall often simply write $\Cgh$ and tacitly assume an
  arbitrary choice of $\alpha$ and $c$ such that
  $\alpha\in(0,2\sqrt{c}/M)$.
\end{remark}

\noindent
For fixed genus $g$, the algebra $\Cgh(\alpha,c)$ is defined by the three parameters
$\hbar,\alpha,c$. However, algebras defined by distinct parameters
might be isomorphic, as shown in the next result.

\begin{proposition}
  If $\C^g_{\hbar_1}(\alpha_1,c_1)$ and
  $\C^g_{\hbar_2}(\alpha_2,c_2)$ are algebras such that
  \begin{equation*}
    \frac{\alpha_1}{\sqrt{c_1}}=\frac{\alpha_2}{\sqrt{c_2}}\mathand
    \hbar_2 = \hbar_1\sqrt{\frac{\alpha_1}{\alpha_2}},
  \end{equation*}
  then
  $\C^g_{\hbar_1}(\alpha_1,c_1)\simeq\C^g_{\hbar_2}(\alpha_2,c_2)$.
\end{proposition}

\begin{proof}
  Let $\lambda=\sqrt{\alpha_1}/\sqrt{\alpha_2}$ giving
  $\hbar_2=\lambda\hbar_1$. We shall prove that
  \begin{align*}
    \varphi(X_1)=X_2\qquad\varphi(Y_1) = \lambda Y_2
    \qquad\varphi(Z_1)=\lambda^2Z_2
  \end{align*}
  defines an isomorphism
  $\varphi:\mathcal{C}^g_{\hbar_1}(\alpha_1,c_1)\to\mathcal{C}^g_{\hbar_2}(\alpha_2,c_2)$. It
  is clear that $\varphi$ is invertible, but to show that it is an
  algebra homomorphism, one needs to prove that it respects the
  relations defining $I_\hbar^g$. First of all, we find that
  \begin{align*}
    &[\varphi(X_1),\varphi(Y_1)]-i\hbar_1\varphi(Z_1)
    = \lambda[X_2,Y_2]-i\hbar_1\lambda^2Z_2 
      =i\lambda\hbar_2Z_2-i\lambda\hbar_2 Z_2=0.
  \end{align*}
  Recalling that
  \begin{align*}
    p_i(x) = \alpha_i\prod_{k=1}^g(x^2-k^2)-\sqrt{c_i}\qquad (i=1,2)
  \end{align*}
  together with $\alpha_1=\lambda^2\alpha_2$ and
  $\sqrt{c_1}=\lambda^2\sqrt{c_2}$, we find that
  \begin{align*}
    p_1(x) = \lambda^2p_2(x)\mathand
    p_1'(x) = \lambda^2p'_2(x).
  \end{align*}
  Now, one can show that
  \begin{align*}
  &[\varphi(Y_1),\varphi(Z_1)]-
      i\hbar_1 p_1\paraa{\varphi(X_1)}p_1'\paraa{\varphi(X_1)}
      -i\hbar_1\T_1\paraa{\varphi(X_1),\varphi(Y_1)^2}\\
    &\qquad
      = \lambda^3[Y_2,Z_2]-i\hbar_1\lambda^{4}p_2(X_2)p_2'(X_2)
      -i\hbar_1\lambda^4\T_2(X_2,Y_2^2)\\
    &\qquad
      = \lambda^3[Y_2,Z_2]-i\hbar_2\lambda^{3}p_2(X_2)p_2'(X_2)
      -i\hbar_2\lambda^3\T_2(X_2,Y_2^2)=0,
  \end{align*}
  as well as
  \begin{align*}
    [\varphi(Z_1),&\varphi(X_1)]
    -2i\hbar_1\varphi(Y_1)^3-i\hbar_1\varphi(Y_1)p_1(\varphi(X_1))
    -i\hbar_1 p_1(\varphi(X_1))\varphi(Y_1)\\
                  &=\lambda^2[Z_2,X_2]-2i\hbar_1\lambda^3Y_1p_2(X_2)
                    -i\hbar_1\lambda^3p_2(X_2)Y_2\\
                  &=\lambda^2[Z_2,X_2]-2i\hbar_2\lambda^2Y_2p_2(X_2)
                    -i\hbar_2\lambda^2p_2(X_2)Y_2 = 0.                        
  \end{align*}
  This proves that $\varphi$ is indeed an algebra homomorphism.
\end{proof}

\noindent
Thus, up to a rescaling of the deformation parameter $\hbar$, the
quotient $\alpha/\sqrt{c}$ is the essential parameter of the algebra
$\Cgh$. This may come as no surprise, since it is the same
quotient that determines the location of the roots of $p(x)$.

In \cite{abhhs:fuzzy,abhhs:noncommutative} the analogous algebra
for the choice
\begin{align*}
  C(x,y,z) = \half(x^2+y^2-\mu)^2+\half z^2-\half c
\end{align*}
was investigated in detail. For $-\sqrt{c}<\mu<\sqrt{c}$ the inverse
image has genus 0, and for $\mu>\sqrt{c}$ the inverse image has genus
1, allowing one to study topology transition by varying the parameter
$\mu$. It turns out that one can classify all finite-dimensional (hermitian)
representations in terms of directed graphs. Curiously, the structure
of the graph clearly reflects the topology of the surface with
``strings'' representing genus 0 and ``cycles'' representing genus
1. It is an interesting question whether or not a similar statement is
true in the higher genus case.

\section{Representations of $\Cgh$}

\noindent
We aim to construct representations $\phi:\Cgh\to\Mat(n,\complex)$
such that $\phi(X)$, $\phi(Y)$ and $\phi(Z)$ are hermitian
matrices. When it is clear from the context, we shall for convenience
simply write $X,Y,Z$ instead of $\phi(X),\phi(Y),\phi(Z)$. Note that
since the matrix algebra generated by a representation is invariant
under hermitian conjugation, every reducible representation will be
completely reducible.

By applying a unitary transformation, one may assume that $X$ is
diagonal with
\begin{align*}
  X = \diag(x_1,x_2,\ldots,x_n)\qquad x_i\in\reals,
\end{align*}
and we note that $Z$ can be eliminated from
\eqref{eq:XYcom}--\eqref{eq:ZXcom} (as $Z=\tfrac{1}{i\hbar}[X,Y]$)
giving
\begin{align}
  &\bracketa{[X,Y],Y} = \hbar^2\paraa{p(X)p'(X)+\T(X,Y^2)}\label{eq:XYYcom}\\
  &\bracketa{[Y,X],X} = \hbar^2\paraa{2Y^3+Yp(X)+p(X)Y}\label{eq:YXXcom}.
\end{align}
Up to now, the choice of $\T(X,Y^2)$ has been quite arbitrary, and it
is time to introduce a few assumptions as well as develop some
notation in the case when $X$ is diagonal. Namely, for an arbitrary
choice of $\T$, every term will be of the form $X^kY^2X^l$ for some
$k,l\geq 0$. If $X$ is assumed to be diagonal, this implies that there
exists a polynomial $\ph(x,y)$ such that
\begin{align*}
  \T(X,Y^2)_{ij} =\ph(x_i,x_j)\sum_{k=1}^ny_{ik}y_{kj},
\end{align*}
giving $\ph(x,x)=p'(x)$.  For instance, if $\T=(p'(X)Y^2+Y^2p'(X))/2$
then
\begin{align}\label{eq:ph.example}
  \ph(x,y) = \half\paraa{p'(x)+p'(y)}.
\end{align}
In the following, we shall assume that $\T$ is chosen such that
\begin{align*}
  \ph(x,y)=\ph(y,x)\mathand \ph(x,-x)=0,
\end{align*}
which is clearly true for \eqref{eq:ph.example} since
$p'(-x)=-p'(x)$. 

The $(i,j)$ matrix elements of \eqref{eq:XYYcom} and \eqref{eq:YXXcom}
can then be written as
\begin{align}
  &\sum_{k=1}^n\paraa{\qh(x_i,x_j)-2x_k}y_{ik}y_{kj}
    = \hbar^2\delta_{ij}p(x_i)p'(x_j)\tag{\Aeq{ij}}\label{eq:Aij}\\
  &2h^2\sum_{k,l=1}^ny_{ik}y_{kl}y_{lj} = \rh(x_i,x_j)y_{ij}\tag{\Beq{ij}}\label{eq:Bij}
\end{align}
where
\begin{align*}
  &\qh(x,y) = x+y-\hbar^2\ph(x,y)\\
  &\rh(x,y) = (x-y)^2-\hbar^2\paraa{p(x)+p(y)}.
\end{align*}
We note that $\qh(x,y)=\qh(y,x)$ and $\rh(x,y)=\rh(y,x)$ as well as
\begin{alignat*}{2}
  &\qh(x,x) = 2x-\hbar^2p'(x) &\qquad &\qh(x,-x)=0\\
  &\rh(x,x) = -2\hbar^2p(x) & &\rh(x,-x) = 4x^2-2\hbar^2p(x).
\end{alignat*}
Since $Y$ is hermitian equation \Aeq{ij} is equivalent to \Aeq{ji}
(due to $\qh$ being symmetric) and equation \Beq{ij} is equivalent to
\Beq{ji}; therefore, it is sufficient to consider \Aeq{ij} and
\Beq{ij} for $i\leq j$. In particular, for $i=j$ one gets
\begin{align*}
  \sum_{k=1}^n\paraa{2(x_i-x_k)-\hbar^2p'(x_i)}|y_{ik}|^2 = \hbar^2p(x_i)p'(x_i).\tag{\Aeq{ii}}
\end{align*}
For a matrix regularization, the matrices $X$, $Y$ and $Z$ correspond
to the the functions $x$, $y$ and $z$, respectively, and 
representations with $\phi(Z)=0$ might be considered degenerate
from this point of view. Therefore, we make the following definition.

\begin{definition}
  A representation $\phi$ of $\Cgh$ is called \emph{degenerate} if
  $\phi(Z)=0$. If a representation is not degenerate, it is called
  \emph{non-degenerate}.
\end{definition}

\noindent
It immediately follows that the structure of degenerate
representations is simple.

\begin{proposition}
  A degenerate representation of $\Cgh$ is completely reducible to a
  sum of 1-dimensional representations.
\end{proposition}

\begin{proof}
  Let $X,Y,Z$ be hermitian matrices of a representation of
  $\Cgh$. Assuming that $Z=0$ it follows that $0 = i\hbar Z = [X,Y]$
  and since $X$ and $Y$ are hermitian, there exists a basis where they
  are both diagonal. Thus, every matrix in the representation is
  diagonal and hence equivalent to a direct sum of 1-dimensional
  representations.
\end{proof}

\subsection{The directed graph of $Y$}

\noindent
Constructing representations of $\Cgh$, in a basis where $X$ is
diagonal, is equivalent to solving equations \eqref{eq:Aij} and
\eqref{eq:Bij} for the matrix elements of $X$ and $Y$. In doing so, it
is convenient to encode the structure of the non-zero matrix elements
of $Y$ in a graph. Therefore, let us recall the concept of a graph
associated to a matrix.
\begin{definition}
  Let $Y$ be a hermitian $(n\times n)$-matrix with matrix elements
  $y_{ij}$ for $i,j\in\{1,\ldots,n\}$. The graph of $Y$ is the (undirected)
  graph $G=(V,E)$ on $n$ vertices such that $(ij)\in E$ if and only if
  $y_{ij}\neq 0$.
\end{definition}

\noindent
Note that the above definition implies that a graph may have
self-loops, i.e. $(ii)\in E$. To distinguish between such a graph,
and a graph without self-loops, we shall refer to the latter as a
\emph{simple graph}. Furthermore, by a \emph{walk} we mean a sequence
of vertices $(i_1i_2\cdots i_N)$ such that $(i_ki_{k+1})\in E$ for
$1\leq k\leq N-1$. A \emph{path} is a walk $(i_1i_2\cdots i_N)$ such
that $i_k\neq i_l$ if $k\neq l$. The \emph{length} of a walk
$(i_1i_2\cdots i_N)$ is $N-1$. For a connected graph G, we let
$d(i,j)$ denote the length of the shortest walk starting at $i\in V$
and ending at $j\in V$.

\begin{definition}
  A graph $G$ is called \emph{admissible} if there exists a
  representation $\phi$ of $\Cgh$ such that $\phi(X)$ is diagonal and
  $G$ is the graph of $\phi(Y)$. In this case, we say that $G$ is the
  graph of $\phi$. If a graph is not admissible, it is called
  \emph{forbidden}.
\end{definition}

\noindent It is easy to see that for a representation to be
irreducible, a necessary condition is that the corresponding graph is
connected.

\begin{proposition}\label{prop:disconnected.graph.irreducible}
  Let $\phi$ be a representation of $\Cgh$ and let $G$ be the graph of
  $\phi$. If $G$ is not connected then $\phi$ is reducible.
\end{proposition}

\begin{proof}
  Let $\phi$ be a representation and choose a basis such that
  $\phi(X)$ is diagonal and $1,2,\ldots,N$ are the vertices of one of
  the components of $G$. Since none of these vertices are connected to
  the remaining vertices of the graph, the matrix $\phi(Y)$ will be
  block diagonal, implying that $\phi(X)$ (which is already diagonal)
  and $Z=[X,Y]/(i\hbar)$ have the same block structure. Hence, $\phi$
  is equivalent to the direct sum of two representations.
\end{proof}

\noindent
The above result implies that one only needs to consider connected
graphs when constructing irreducible representations.

\begin{lemma}\label{lemma:graph.path.length.three}
  Let $G=(V,E)$ and let $i,j\in V$ such that there
  exists a unique walk of length 3 from $i$ to $j$. If
  $(ij)\notin E$ then $G$ is forbidden.
\end{lemma}

\begin{proof}
  Let $(iklj)$ denote the unique walk of length 3 such
  that $y_{ik},y_{kl},y_{lj}\neq 0$. Equation \eqref{eq:Bij} gives
  \begin{align*}
    2\hbar^2y_{ik}y_{kl}y_{lj} = \rh(x_i,x_j)y_{ij}=0
  \end{align*}
  since $(ij)\notin E$. But this contradicts the assumption that these
  matrix elements are non-zero. Hence, $G$ is forbidden.
\end{proof}

\noindent
The above result has immediate consequences that exclude certain
classes of graphs from being representations.

\begin{corollary}
  Let $G=(V,E)$ be a tree. If there exist $i,j\in E$ such that
  $d(i,j)\geq 3$ then $G$ is forbidden. Hence, any admissible tree
  is ``star shaped''.
\end{corollary}

\begin{proof}
  Assume that there exists a pair of vertices $i,j$ such that
  $d(i,j)\geq 3$, which implies that there exists a vertex $k$ such
  that $d(i,k)=3$. Since $G$ is a tree there is exactly one path
  realizing this distance, and there is no other path connecting $i$
  and $k$. In particular, $(ik)\notin E$.  Then
  Lemma~\ref{lemma:graph.path.length.three} implies that $G$ is
  forbidden.
\end{proof}

\begin{corollary}\label{cor:forbidden.cycles}
  Let $G=C_n$ be the cycle graph on $n\geq 3$ vertices. If $n\notin\{4,6\}$
  then $G$ is forbidden.
\end{corollary}

\begin{proof}
  For $n\geq 7$ let $i$ and $j$ be vertices in the cycle with
  $d(i,j)=3$. Since $(ij)\notin E$,
  Lemma~\ref{lemma:graph.path.length.three} implies that $G$ is
  forbidden. For $n=5$, let $i,j$ be vertices with $d(i,j)=3$. Then it
  is easy to check that there is precisely one other path between $i$
  and $j$, and that path is of length two. Thus,
  Lemma~\ref{lemma:graph.path.length.three} implies that $G$ is
  forbidden. Similarly, for $n=3$, there is exactly one walk of length
  $3$ from a vertex $i$ to itself. Since $(ii)\notin V$,
  Lemma~\ref{lemma:graph.path.length.three} implies that $G$ is
  forbidden.
\end{proof}

\noindent
The above results indicate that a general admissible graph probably
has a dense edge structure without large sparse subgraphs. Let us now
continue to study representations of low dimensions.

\subsection{1-dimensional representations}

\noindent
Consider the case when $X,Y$ are $1\times 1$ matrices, and write
$X=x\in\reals$, $Y=y\in\reals$ and $Z=z\in\reals$. Equations
\eqref{eq:XYYcom} and \eqref{eq:YXXcom} become
\begin{align*}
  &p'(x)\paraa{p(x)+y^2} = 0\\
  &y\paraa{p(x)+y^2} = 0
\end{align*}
Clearly, there are solutions with $p(x)=-y^2$. That is, for every
$x\in\reals$ such that $p(x)<0$ one sets $y=\sqrt{|p(x)|}$. If
$p(x)+y^2\neq 0$ then one must necessarily have $y=0$ and
$p'(x)=0$. Note that $Z=0$ since $X$ and $Y$ commute, implying that
all one dimensional representations are degenerate.

\subsection{2-dimensional representations}

In this section, we shall construct irreducible 2-dimensional
representations of $\Cgh$. As previously noted,
Proposition~\ref{prop:disconnected.graph.irreducible} implies that one
only needs to consider connected graphs, and there are 3
non-isomorphic connected graphs with two vertices as shown in
Figure~\ref{fig:2dgraphs}. It follows immediately from
Lemma~\ref{lemma:graph.path.length.three} that the graph of Type III
is forbidden since there is a unique walk of length 3 from the vertex
with no self-loop to itself. In the following, we shall see that both
graphs of Type I and II are in fact admissible.
\begin{figure}[h]
  \centering
  \includegraphics[width=11cm]{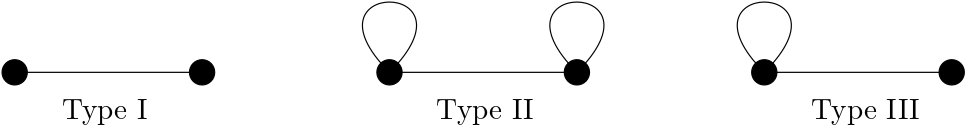}
  \caption{Connected graphs with 2 vertices.}
  \label{fig:2dgraphs}
\end{figure}

\noindent
For 2-dimensional representations, one can slightly strengthen the
correspondence between representations and graphs. Before formulating
this result, let us prove the following lemma which will later also be
useful when discussing equivalence of representations.

\begin{lemma}\label{lemma:unitary.2x2}
  Let $x_1\neq x_2$,
  \begin{align*}
    X =
    \begin{pmatrix}
      x_1 & 0 \\ 0 & x_2
    \end{pmatrix}\qquad
                     \tilde{X} =
                     \begin{pmatrix}
                       \tilde{x}_1 & 0 \\ 0 & \tilde{x}_2
                     \end{pmatrix}
  \end{align*}
  and assume that there exists a unitary matrix $U$ such that
  $UXU^\dagger = \tilde{X}$. For an arbitrary hermitian matrix
  \begin{align*}
    Y=
    \begin{pmatrix}
      y_1 & z \\ \bar{z} & y_2
    \end{pmatrix}
  \end{align*}
  it follows that there exists $\varphi\in\reals$ such that either
  \begin{align*}
    \tilde{X} =
    \begin{pmatrix}
      x_1 & 0 \\ 0 & x_2
    \end{pmatrix}\mathand
                     UYU^\dagger =
                     \begin{pmatrix}
                       y_1 & e^{i\varphi}z\\
                       e^{-i\varphi}z & y_2
                     \end{pmatrix}
  \end{align*}
  or
  \begin{align*}
    \tilde{X} =
    \begin{pmatrix}
      x_2 & 0 \\ 0 & x_1
    \end{pmatrix}\mathand
                     UYU^\dagger =
                     \begin{pmatrix}
                       y_2 & e^{-i\varphi}\bar{z}\\
                       e^{i\varphi}z & y_1
                     \end{pmatrix}.                                       
  \end{align*}
\end{lemma}

\begin{proof}
  Writing an arbitrary unitary $2\times 2$ matrix as
  \begin{align*}
    e^{\tilde{\varphi}/2}
    \begin{pmatrix}
      e^{i\varphi_1}\cos\theta & e^{i\varphi_2}\sin\theta \\
      -e^{-i\varphi_2}\sin\theta & e^{-i\varphi_1}\cos\theta
    \end{pmatrix}
  \end{align*}
  one finds that
  \begin{align*}
    UXU^\dagger=
    \begin{pmatrix}
      x_1\cos^2\theta + x_2\sin^2\theta &
      -\tfrac{1}{2}e^{i(\varphi_1+\varphi_2)}(x_1-x_2)\sin 2\theta\\
      -\tfrac{1}{2}e^{i(\varphi_1+\varphi_2)}(x_1-x_2)\sin 2\theta
      & x_1\sin^2\theta + x_2\cos^2\theta\\
    \end{pmatrix}.
  \end{align*}
  For this matrix to be diagonal, a necessary condition is that
  $\sin 2\theta=0$ since $x_1\neq x_2$. Thus,
  $\theta=0,\pi/2,\pi,3\pi/2$ and it is easy to see that $UYU^\dagger$
  has the required form for these values of $\theta$.
\end{proof}

\noindent
Using the above result, one can prove that unitarily equivalent
representations have isomorphic graphs.

\begin{proposition}\label{prop:2d.graph.rep.iso}
  If $\phi$ and $\phi'$ are unitarily equivalent non-degenerate
  2-dimensional representations of $\Cgh$, then the graph of $\phi$ is
  isomorphic to the graph of $\phi'$.
\end{proposition}

\begin{proof}
  If $\phi$ is a non-degenerate representation with
  $\phi(X)=\diag(x_1,x_2)$ then one must necessarily have
  $x_1\neq x_2$ since, otherwise, $\phi(X)$ commutes with $\phi(Y)$
  (giving $\phi(Z)=0$). Then one may apply
  Lemma~\ref{lemma:unitary.2x2} to conclude that $\phi(Y)$ and
  $\phi'(Y)$ have exactly the same structure of non-zero matrix
  elements (up to a permutation of the rows and columns) implying that
  the corresponding graphs are isomorphic.
\end{proof}

\noindent
In particular, it follows from the above result that representations
with graphs of Type I and Type II are inequivalent.  Now, let us turn
to the task of finding concrete representations.

For a representation $\phi$ with
\begin{align*}
  \phi(X) =
  \begin{pmatrix}
    x_1 & 0 \\ 0 & x_2
  \end{pmatrix}\qquad
                   \phi(Y)=
                   \begin{pmatrix}
                     y_1 & z \\ \bar{z} & y_2
                   \end{pmatrix}
\end{align*}
equations \Aeq{ij} and \Beq{ij} become
\begin{align*}
  \paraa{&2(x_1-x_2)-\hbar^2p(x_1)}|z|^2-\hbar^2p'(x_1)y_1^2
           =\hbar^2p(x_1)p'(x_1)\tag{\Aeq{11}}\\
  \paraa{-&2(x_1-x_2)-\hbar^2p(x_2)}|z|^2-\hbar^2p'(x_2)y_2^2
            =\hbar^2p(x_2)p'(x_2)\tag{\Aeq{22}}\\
         &z\paraa{\qh(x_1,x_2)(y_1+y_2)-2x_1y_1-2x_2y_2} = 0\tag{\Aeq{12}}\\
         &y_1^3+(2y_1+y_2)|z|^2+p(x_1)y_1 = 0\tag{\Beq{11}}\\
         &y_2^3+(2y_2+y_1)|z|^2+p(x_2)y_2 = 0\tag{\Beq{22}}\\
         &y_1^2+y_2^2+y_1y_2+|z|^2=\tfrac{1}{2\hbar^2}\rh(x_1,x_2)\tag{\Beq{12}}.
\end{align*}
If $z\neq 0$ and $x_1=-x_2=x\neq 0$ the above equations are equivalent
to
\begin{align}
  &\paraa{4x-\hbar^2p(x)}|z|^2-\hbar^2p'(x)y^2=\hbar^2p(x)p'(x)\label{eq:2d.xmx.A11}\\ 
  &y\paraa{y^2+3|z|^2+p(x)} = 0\label{eq:2d.xmx.B11}\\ 
  &3y^2+|z|^2 = \frac{2x^2}{\hbar^2}-p(x)\label{eq:2d.xmx.B12}
\end{align}
with $y_1=y_2=y$. Let us start by considering representations of Type I.

\begin{proposition}\label{prop:2dim.rep.string}
  $\phi$ is a non-degenerate representation of $\Cgh$ such that
  \begin{align*}
    \phi(X) =
    \begin{pmatrix}
      \xh & 0 \\
      0 & -\xh
    \end{pmatrix}\qquad
    \phi(Y)=
    \begin{pmatrix}
      0 & z \\ \bar{z} & 0
    \end{pmatrix}\qquad
                         \phi(Z)=\frac{2i}{\hbar}
                         \begin{pmatrix}
                           0 & -\xh z\\
                           \xh \bar{z} & 0
                         \end{pmatrix}    
  \end{align*}
  if and only if
  \begin{align*}
    z = \frac{1}{\hbar}e^{i\theta}\sqrt{2\hat{x}^2-\hbar^2p(\hat{x})}
  \end{align*}
  for $\theta\in\reals$ and $\xh$ satisfying
  \begin{align}
      &2p(\xh)+\xh p'(\xh)-\frac{4\xh^2}{\hbar^2} = 0\label{eq:2dim.x.p.pp}\\
      &2\xh^2-\hbar^2p(\xh)>0.\label{eq:2dim.x.p.geq}
  \end{align}
\end{proposition}

\begin{proof}
  Let $X$ and $Y$ be matrices of a 2-dimensional non-degenerate
  representation of $\Cgh$ of the form above. Since $Z\neq 0$ we
  necessarily have $\xh,z\neq 0$. Equations
  \eqref{eq:2d.xmx.A11}--~\eqref{eq:2d.xmx.B12} (with $y=0$) then
  become
  \begin{align*}
    \paraa{4x-\hbar^2p'(x)}|z|^2 &= \hbar^2p(x)p'(x)\\
    |z|^2 &= \frac{2x^2}{\hbar^2}-p(x)
  \end{align*}
  which are equivalent to
  \begin{align*}
      &2p(x)+xp'(x)-\frac{4x^2}{\hbar^2} = 0\\
      &|z|^2 = \frac{2x^2}{\hbar^2}-p(x).
  \end{align*}
  The fact that $\phi$ is non-degenerate implies that $z\neq 0 $ which
  necessarily gives $\frac{2x^2}{\hbar^2}-p(x)>0$, proving the first
  part of the statement.

  Conversely, if $\xh$ is a solution of \eqref{eq:2dim.x.p.pp} such
  that $\frac{2\xh^2}{\hbar^2}-p(\xh)>0$ then one may construct a solution
  by setting
  \begin{align*}
    z = \frac{1}{\hbar}e^{i\theta}\sqrt{2\hat{x}^2-\hbar^2p(\hat{x})}\neq 0
  \end{align*}
  for arbitrary $\theta\in\reals$. Finally, to prove that $\phi$ is
  non-degenerate (i.e. $Z\neq 0)$, we need to show that $\xh=0$ is
  never a solution to \eqref{eq:2dim.x.p.pp} and
  \eqref{eq:2dim.x.p.geq}.  If $\xh=0$ is a solution to
  \eqref{eq:2dim.x.p.pp} then $p(\xh)=0$, which implies that
  $\frac{2\xh^2}{\hbar^2}-p(\xh)=0$, which does not fulfill
  \eqref{eq:2dim.x.p.geq}. Hence, any solution to
  \eqref{eq:2dim.x.p.pp} and \eqref{eq:2dim.x.p.geq} is non-zero.
\end{proof}

\noindent
The next result ensures that one may find a Type I representation of
$\Cgh$ for any value of the deformation parameter $\hbar>0$.

\begin{proposition}\label{prop:f.tau.solution}
  For $\hbar>0$ and $g\geq 2$, there exists $\xh>g-1$ such that
  \begin{align*}
    2p(\xh)+\xh p'(\xh)-\frac{4\xh^2}{\hbar^2}
    =0\mathand \frac{2\xh^2}{\hbar^2}-p(\xh)>0.
  \end{align*}
\end{proposition}

\begin{proof}
  First we note that if $2p(\xh)+\xh p'(\xh)-4\xh^2/\hbar^2=0$ and $\xh p'(\xh)>0$ then
  \begin{align*}
    \frac{2\xh^2}{\hbar^2}-p(\xh) =
    \frac{2\xh^2}{\hbar^2}+\frac{1}{2}\xh p'(\xh)-\frac{2\xh^2}{\hbar^2}
    =\frac{1}{2}\xh p'(\xh) > 0.
  \end{align*}
  Writing $\fh(x) = 2p(x)+xp'(x)-4x^2/\hbar^2$ one finds that
  \begin{align*}
    \fh(g-1) = -2\sqrt{c} + (g-1)p'(g-1)-\frac{4(g-1)^2}{\hbar^2}<0
  \end{align*}
  since $p'(g-1)<0$ (cf. \eqref{eq:pp.integer.value}).
  For $g\geq 2$, $\fh(x)$ is a polynomial of degree $2g$ with a
  positive coefficient of $x^{2g}$, which implies that $\fh(x)$ is
  positive for large $x$. In particular, there exists $\xh>g-1$ such that
  $\fh(\xh)=0$. If $\xh\geq g$ then $\xh p'(\xh)>0$ since $xp'(x)>0$
  for all $x\geq g$. If $\xh\in(g-1,g)$ and $\fh(\xh)=0$ then
  $\xh p'(\xh)>0$ due to the fact that $p(x)<0$ for all $x\in(g-1,g)$ and
  \begin{align*}
    \xh p'(\xh) = \frac{4\xh^2}{\hbar^2}-2p(\xh).
  \end{align*}
  From the argument in the beginning of the proof, it follows that
  \begin{equation*}
    2\xh^2/\hbar^2-p(\xh)>0.\qedhere
  \end{equation*}
\end{proof}

\noindent
In the case when $g=1$, one obtains
\begin{align*}
  2p(x)+xp'(x)-\frac{4x^2}{\hbar^2}
  = 4x^2\parab{\alpha-\frac{1}{\hbar^2}}-2(\alpha+\sqrt{c})
\end{align*}
which does not have any real solutions for small enough
$\hbar$. However, this case has been treated thoroughly in
\cite{abhhs:fuzzy,abhhs:noncommutative}.

In any case, we have shown the existence of representations for
arbitrary $g\geq 2$ and values of the deformation parameter
$\hbar$. Let us state this result as follows.

\begin{corollary}
  For $g\geq 2$, $\hbar>0$, $c>0$ and $\alpha\in(0,2\sqrt{c}/M)$,
  there exists a 2-dimensional non-degenerate representation of
  $\Cgh(\alpha,c)$.
\end{corollary}

\noindent
Next, we consider representations of Type II.

\begin{proposition}\label{prop:2d.typeII}
  $\phi$ is a non-degenerate representation of $\Cgh$ such that
  \begin{align*}
    \phi(X) =
    \begin{pmatrix}
      \xh & 0 \\
      0 & -\xh
    \end{pmatrix}\qquad
    \phi(Y)=
    \begin{pmatrix}
      y_1 & z \\ \bar{z} & y_2
    \end{pmatrix}\qquad
                         \phi(Z)=\frac{2i}{\hbar}
                         \begin{pmatrix}
                           0 & -\xh z\\
                           \xh\bar{z} & 0
                         \end{pmatrix}        
  \end{align*}
  with $y_1\neq 0$, if and only if
  \begin{align*}
    y_1 = y_2 = \pm\frac{1}{2\hbar}&\sqrt{3\xh^2-\hbar^2p(\xh)}\\
    z =  \frac{1}{2\hbar}e^{i\theta}&\sqrt{-\xh^2-\hbar^2p(\xh)}
  \end{align*}
  for $\theta\in\reals$ and $\xh\neq 0$ satisfying
  \begin{align}
    &2p(\xh)-\xh p'(\xh)<0\label{eq:2d.typeII.prop.xpeq}\\
    &2\xh+\hbar^2p'(\xh) = 0.\label{eq:2d.typeII.prop.ineq}
  \end{align}
\end{proposition}

\begin{proof}
  Assume that $X$ and $Y$ are matrices of a non-degenerate
  2-dimensional representation of the form above with $y_1\neq
  0$. Since $Z\neq 0$, we necessarily have that $\xh,z\neq 0$. As
  already noted, when $x_1=-x_2$ and $z\neq 0$, equation \Aeq{12}
  implies that $y_1=y_2=y$. Since $y=y_1\neq 0$, equations
  \eqref{eq:2d.xmx.A11}--\eqref{eq:2d.xmx.B12} become
  \begin{align*}
    &\paraa{4x-\hbar^2p(x)}|z|^2-\hbar^2p'(x)y^2=\hbar^2p(x)p'(x)\\
    &y^2+3|z|^2+p(x) = 0\\
    &3y^2+|z|^2 = \frac{2x^2}{\hbar^2}-p(x)
  \end{align*}
  which are equivalent to
  \begin{align}
    &\paraa{2x+\hbar^2p'(x)}\paraa{x^2+\hbar^2p(x)} = 0\label{eq:2d.typeII.factoreq}\\
    &y^2 = \frac{1}{4\hbar^2}\paraa{3x^2-\hbar^2p(x)}\label{eq:2d.typeII.ysq}\\
    &|z|^2 =-\frac{1}{4\hbar^2}\paraa{x^2+\hbar^2 p(x)}.\label{eq:2d.typeII.zsq}
  \end{align}
  From \eqref{eq:2d.typeII.factoreq} it follows that either
  $2x+\hbar^2p'(x)=0$ or $x^2+\hbar^2p(x)=0$. However, if
  $x^2+\hbar^2p(x)=0$ then $z=0$, which contradicts the assumption
  that $\phi$ is non-degenerate. Hence, it must hold that
  $2\xh+\hbar^2p'(\xh)=0$. Moreover, from \eqref{eq:2d.typeII.zsq} it
  follows that $\xh^2+\hbar^2p(\xh)<0$ which, by inserting
  $\xh^2=-\tfrac{1}{2}\hbar^2\xh p'(\xh)$ gives $2p(\xh)-\xh p'(\xh)<0$.
  Conversely, assume that \eqref{eq:2d.typeII.prop.xpeq} and
  \eqref{eq:2d.typeII.prop.ineq} holds for some $\xh\in\reals$, giving
  \begin{align*}
    \xh p'(\xh) = -\frac{2\xh^2}{\hbar^2}\mathand
    \xh^2+\hbar^2p(\xh)<0,
  \end{align*}
  implying that \eqref{eq:2d.typeII.zsq} is satisfied by defining
  \begin{align*}
    z =  \frac{1}{2\hbar}e^{i\theta}&\sqrt{|\xh^2+\hbar^2p(\xh)|}\neq 0
  \end{align*}
  for arbitrary $\theta\in\reals$.  Moreover, since
  $\xh^2+\hbar^2p(\xh)<0$ it follows that
  \begin{align*}
    3\xh^2-\hbar^2p(\xh)>3\xh^2+\xh^2=4\xh^2>0,
  \end{align*}
  implying that
  \begin{align*}
    y=\pm\frac{1}{2\hbar}&\sqrt{3\xh^2-\hbar^2p(\xh)}\neq 0
  \end{align*}
  solves \eqref{eq:2d.typeII.ysq}. Finally, equation
  \eqref{eq:2d.typeII.factoreq} is satisfied since
  $2\xh+\hbar^2p'(\xh)=0$. The representation defined in this way will
  be non-degenerate since $\xh\neq 0$ (by assumption) and $z\neq 0$
  due to $\xh^2+\hbar^2p(\xh)<0$.
\end{proof}

\noindent
Thus, to construct a representation of Type II as in
Proposition~\ref{prop:2d.typeII} one needs to find $\xh$ such that
\begin{align*}
  2p(\xh)-\xh p'(\xh)<0\\
  2\xh+\hbar^2p'(\xh) = 0.
\end{align*}
It is useful to think of the second equation as determining
$\hbar$. That is, for any $\xh\in\reals$ such that $p'(\xh)<0$, we can
consider Type II representations of $\Cgh$ with
\begin{align*}
  \hbar=\sqrt{-\frac{2\xh}{p'(\xh)}}
\end{align*}
satisfying $2\xh+\hbar^2p'(\xh)=0$.  The next result guarantees one
may find such an $\xh$ also satisfying $2p(\xh)-\xh p'(\xh)<0$.

\begin{proposition}
  If $\alpha/\sqrt{c}<2/(2g-1)!$ then
  \begin{align*}
    p'(g-1)<0\mathand 2p(g-1)-(g-1)p'(g-1)<0.
  \end{align*}
\end{proposition}

\begin{proof}
  One notes immediately that $p'(g-1)<0$ from
  \eqref{eq:pp.integer.value}. Moreover, one obtains
  \begin{align*}
    2p(g-1)-(g-1)p'(g-1) &= -2\sqrt{c}+\alpha(2g-1)!\\
    &=\sqrt{c}(2g-1)!\parab{\frac{\alpha}{\sqrt{c}}-\frac{2}{(2g-1)!}}<0
  \end{align*}
  by assumption.
\end{proof}

\noindent
Hence, for $\alpha/\sqrt{c}<2/(2g-1)!$ one may construct a Type II
representation of $\Cgh$ as in Proposition~\ref{prop:2d.typeII} with
\begin{align*}
  \hbar = \sqrt{-\frac{2(g-1)}{p'(g-1)}} = \sqrt{\frac{2(g-1)^2}{\alpha(2g-1)!}}
\end{align*}
and $\xh=g-1$, giving
\begin{align*}
  &y = \pm\frac{1}{2}\sqrt{\sqrt{c}+\tfrac{3}{2}\alpha(2g-1)!}\\
  &z = \frac{1}{2}e^{i\theta}\sqrt{\sqrt{c}-\tfrac{1}{2}\alpha(2g-1)!}.
\end{align*}
The above considerations illustrate a property which is
generic in the context of matrix regularizations. Namely, the
deformation parameter $\hbar$ is related to the dimension of the
representation giving restrictions on for which $\hbar$ an
$N$-dimensional representation may exist. 

It is clear from Proposition~\ref{prop:2d.graph.rep.iso} that
representations of Type I and II are not equivalent. However, we would
like to investigate to what extent there are inequivalent
representations of the same type.

\begin{proposition}\label{prop:2d.typeI.equivalence}
  Let $\phi$ and $\phi'$ be representations as in
  Proposition~\ref{prop:2dim.rep.string} with
  \begin{align*}
    \phi(X) =
    \begin{pmatrix}
      \xh & 0 \\ 0 & -\xh
    \end{pmatrix}\mathand
                     \phi'(X) =
                     \begin{pmatrix}
                       \xh' & 0 \\ 0 & -\xh'
                     \end{pmatrix}.
  \end{align*}
  Then $\phi$ is unitarily equivalent to $\phi'$ if and only if
  $\xh=\pm\xh'$.
\end{proposition}

\begin{proof}
  First, assume that $\phi$ and $\phi'$ are unitarily equivalent
  representations as in Proposition~\ref{prop:2dim.rep.string}. Since
  the diagonal elements of $\phi(X)$ are distinct (due to the
  assumption that $\phi$ is non-degenerate) one may apply
  Lemma~\ref{lemma:unitary.2x2} to conclude that either $\xh'=\xh$ or
  $\xh'=-\xh$.

  Next, assume that $\phi$ and $\phi'$ are representations as in
  Proposition~\ref{prop:2dim.rep.string} such that
  $\xh'=\xh$ and
  \begin{align*}
    &\phi(Y) = \frac{1}{\hbar}\sqrt{2\xh^2-\hbar^2p(\xh)}
    \begin{pmatrix}
      0 & e^{i\theta} \\ e^{-i\theta} & 0
    \end{pmatrix}\\
    &\phi'(Y) = \frac{1}{\hbar}\sqrt{2\xh^2-\hbar^2p(\xh)}
    \begin{pmatrix}
      0 & e^{i\theta'} \\ e^{-i\theta'} & 0
    \end{pmatrix}.                                       
  \end{align*}
  It is easy to see that these matrices are indeed unitarily
  equivalent with $\phi'(Y)=U\phi(Y)U^\dagger$ for
  $U=\diag(e^{i\theta'},e^{i\theta})$. The argument for $\xh'=-\xh$ is
  analogous and uses the fact that $p(-\xh)=p(\xh)$.
\end{proof}

\noindent
The result for representations of Type II is similar.

\begin{proposition}\label{prop:2d.typeII.equivalence}
  Let $\phi$ and $\phi'$ be representations as in
  Proposition~\ref{prop:2d.typeII} with
  \begin{alignat*}{2}
    &\phi(X) =
    \begin{pmatrix}
      \xh & 0 \\ 0 & -\xh
    \end{pmatrix} &\qquad
    &\phi(Y) =
    \begin{pmatrix}
      y & z \\ \bar{z} & y
    \end{pmatrix}\\    
    &\phi'(X) =
    \begin{pmatrix}
      \xh' & 0 \\ 0 & -\xh'
    \end{pmatrix} &\qquad
    &\phi'(Y) =
    \begin{pmatrix}
      y' & z' \\ \bar{z}' & y'
    \end{pmatrix}.    
  \end{alignat*}
  Then $\phi$ is unitarily equivalent to $\phi'$ if and only if
  $\xh=\pm\xh'$ and $y=y'$.
\end{proposition}

\begin{proof}
  The proof is in complete analogy with the proof of
  Proposition~\ref{prop:2d.typeI.equivalence} and will not be repeated
  in detail. The only slight difference is that there is a choice of
  sign in $y$ which can not be compensated for by a unitary
  transformation.
\end{proof}

\subsection{A 3-dimensional representation}

\begin{figure}[h]
  \centering
  \includegraphics[width=55mm]{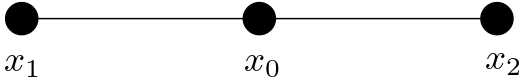}
  \caption{An admissible graph on three vertices.}
  \label{fig:3dstring}
\end{figure}

\noindent
Let us give an example of a 3-dimensional representation with a graph
as in Figure~\ref{fig:3dstring}.  Equations \Aeq{ij} and \Beq{ij} become
\begin{align*}
  \paraa{2(x_0-x_1)&-\hbar^2p'(x_0)}|z_1|^2
                     +\paraa{2(x_0-x_2)-\hbar^2p'(x_0)}|z_2|^2=\hbar^2p(x_0)p'(x_0)\tag{\Aeq{00}}\\
                   &\paraa{2(x_1-x_0)-\hbar^2p'(x_1)}|z_1|^2=\hbar^2p(x_1)p'(x_1)\tag{\Aeq{11}}\\
                   &\paraa{2(x_2-x_0)-\hbar^2p'(x_2)}|z_2|^2=\hbar^2p(x_2)p'(x_2)\tag{\Aeq{22}}\\
                   &x_0 = \frac{1}{2}\qh(x_1,x_2)\tag{\Aeq{12}}\\
                   &2\hbar^2\paraa{|z_1|^2+|z_2|^2}=\rh(x_0,x_1)\tag{\Beq{01}}\\
                   &2\hbar^2\paraa{|z_1|^2+|z_2|^2}=\rh(x_0,x_2)\tag{\Beq{02}}.
\end{align*}
where $z_1=y_{01}$ and $z_2=y_{02}$.  For $x_0=0$, $x_1=-x_2=x$ one
uses $\qh(x,-x)=0$ and $p'(0)=0$ to show that the equations are
equivalent to
\begin{align*}
  &r=|z_1|=|z_2|\tag{\Aeq{00}}\\
  &\paraa{2x-\hbar^2p'(x)}r^2=\hbar^2p(x)p'(x)\tag{\Aeq{11}}\\
  &4\hbar^2r^2 = x^2-\hbar^2\paraa{p(0)+p(x)}\tag{\Beq{01}}.
\end{align*}

\begin{proposition}\label{prop:3d.string.rep}
  If $\xh^2-\hbar^2\paraa{p(0)+p(\xh)}>0$, $\xh\neq 0$ and 
  \begin{align*}
    \parac{\frac{2\xh}{\hbar^2}-p'(\xh)}\parac{\frac{\xh^2}{\hbar^2}-p(0)-p(\xh)}=4p(\xh)p'(\xh)
  \end{align*}
  then
  \begin{align*}
    \phi(X) =
    \begin{pmatrix}
      0 & 0 & 0\\
      0 & \xh & 0 \\
      0 & 0 & -\xh
    \end{pmatrix}\quad
    \phi(Y) =
    \begin{pmatrix}
      0 & z_1 & z_2 \\
      \bar{z}_1 & 0 & 0\\
      \bar{z}_2 & 0 & 0
    \end{pmatrix}\quad
                      \phi(Z) =\frac{i\xh}{\hbar}
                      \begin{pmatrix}
                        0 & z_1 & -z_2 \\
                        -\bar{z}_1 & 0 & 0 \\
                        \bar{z}_2 & 0 & 0 
                      \end{pmatrix}
  \end{align*}
  with
  \begin{align*}
    &z_1 = \frac{1}{2\hbar}e^{i\theta_1}\sqrt{\xh^2-\hbar^2\paraa{p(\xh)+p(0)}}\\
    &z_2 = \frac{1}{2\hbar}e^{i\theta_2}\sqrt{\xh^2-\hbar^2\paraa{p(\xh)+p(0)}},
  \end{align*}
  for $\theta_1,\theta_2\in\reals$, define a non-degenerate representation of $\Cgh$.
\end{proposition}

\noindent
There are several ways of satisfying the requirements of
Proposition~\ref{prop:3d.string.rep}, and let us give a particular
construction in the next result. To this end, let us introduce the following notation
\begin{align*}
  &\fh(x) = \parac{\frac{2x}{\hbar^2}-p'(x)}\parac{\frac{x^2}{\hbar^2}-p(0)-p(x)}-4p(x)p'(x)\\
  &\rh(x) = \frac{x^2}{4\hbar^2}-\frac{1}{4}\paraa{p(0)+p(x)}.
\end{align*}

\begin{proposition}\label{prop:3d.hbar.solution}
  There exists $\hbar>0$ such that
  \begin{align*}
    \fh(g-1)=0\mathand \rh(g-1)>0.
  \end{align*}
\end{proposition}

\begin{proof}
  Expanding $\fh$ gives
  \begin{align*}
    \fh(x) = \frac{2x^3}{\hbar^4}-\frac{1}{\hbar^2}\paraa{p(0)+p(x)+p'(x)}+p'(x)\paraa{p(0)-3p(x)} 
  \end{align*}
  If $x>0$ then $\fh(x)>0$ for small enough $\hbar$, and if
  $p'(x)\paraa{p(0)-3p(x)}<0$ then $\fh(x)<0$ for large enough
  $\hbar$. Thus, for such an $x$, there exists $\hbar>0$ such that
  $\fh(x)=0$. Let us now show that $x=g-1$ fulfills
  \begin{align*}
    p'(g-1)\paraa{p(0)-3p(g-1)}=p'(g-1)\paraa{p(0)+3\sqrt{c}}<0.
  \end{align*}
  Since $p'(g-1)<0$ the above is equivalent to $p(0)+3\sqrt{c}>0$
  which is true by Lemma~\ref{lemma:p.zero.c}. Next, let us show that
  $\rh(g-1)>0$ for all $\hbar>0$:
  \begin{align*}
    \rh(g-1)=\frac{(g-1)^2}{4\hbar^2}-\frac{1}{4}\paraa{p(0)-\sqrt{c}}>0
  \end{align*}
  since Lemma~\ref{lemma:p.zero.c} gives $p(0)-\sqrt{c}<0$.
\end{proof}

\noindent
Thus, for $\xh=g-1$ and $\hbar$ as in
Proposition~\ref{prop:3d.hbar.solution} we find that
\begin{align*}
  &z_1 = \frac{1}{2\hbar}e^{i\theta_1}\sqrt{(g-1)^2-\hbar^2
    \paraa{(-1)^g\alpha(g!)^2-2\sqrt{c}}}\\
  &z_2 = \frac{1}{2\hbar}e^{i\theta_2}\sqrt{(g-1)^2-\hbar^2
    \paraa{(-1)^g\alpha(g!)^2-2\sqrt{c}}}.
\end{align*}
Furthermore, it is clear that the phases $\theta_1$ and $\theta_2$ are
inessential as one may always find an equivalent representation with
$\theta_1=\theta_2=0$ by conjugating with a diagonal unitary matrix.
Finally, let us show that these representations are irreducible.

\begin{proposition}
  If $\phi$ is a representation of $\Cgh$ as in
  Proposition~\ref{prop:3d.string.rep} then $\phi$ is irreducible.
\end{proposition}

\begin{proof}
  If $\phi$ is reducible, then $\phi$ is completely reducible to a sum
  of lower dimensional representations, which implies that $\phi$ is
  equivalent to a representation with at least two disconnected
  components in the corresponding graph. Since $\xh\neq 0$ the three
  eigenvalues of $X$ are distinct, which implies that unitary
  matrices $U$, such that $UXU^\dagger$ is again diagonal, are
  composed of permutations and diagonal unitary matrices. Hence, the
  graph of $Y$ is isomorphic to the graph of $UYU^\dagger$, which
  implies that $UYU^\dagger$ is connected. We conclude that $\phi$ is
  irreducible.
\end{proof}

\section{Concluding remarks}

\noindent Finding matrix regularizations for surfaces of higher genus
seems to be a notoriously difficult problem. The simple structure of
representations in the case of genus 0 and 1, does not lend itself to
any easy generalizations. In this paper, we have followed the idea of
finding representations of a one-parameter family of deformations of a
commutative algebra (corresponding to smooth functions on the
surface), in order to generate a matrix regularization. These algebras
have a concrete definition in terms of generators and relations in
direct correspondence with the definition of surfaces as level sets in
$\reals^3$. Explicit low-dimensional representations have been
constructed, but ideally one would like a complete sequence of matrix
algebras (of increasing dimension) for a sequence of the deformation
parameter $\hbar$ tending to zero. Although we have not reached our
final goal, we believe that our investigation has considerably
increased the understanding of representations, as well as the
analytic structure of the defining polynomials, paving the way for
future work.

\section*{Acknowledgments}

\noindent
J.A would like to thank M. Aigner and A. Sykora for discussions and
the Swedish Research Council for financial support.

\bibliographystyle{alpha}
\bibliography{references}  

\end{document}